\documentclass{amsart}
\usepackage[T1]{fontenc}
\usepackage{amssymb}
\usepackage{microtype}
\usepackage[hidelinks]{hyperref}

\newtheorem{theorem}{Theorem}
\newtheorem{lemma}[theorem]{Lemma}
\newtheorem{proposition}[theorem]{Proposition}
\newtheorem{conjecture}[theorem]{Conjecture}

\newcommand{\Prob}{\mathsf P}
\newcommand{\E}{\mathsf E}
\newcommand{\R}{\mathbb R}
\DeclareMathOperator{\tr}{tr}
\DeclareMathOperator{\diag}{diag}
\DeclareMathOperator{\adj}{adj}
\DeclareMathOperator{\Var}{Var}
\newcommand{\smin}{s_{\min}}
\newcommand{\norm}[1]{\lVert#1\rVert}
\newcommand{\ip}[2]{\langle#1,#2\rangle}

\title[Least singular values with deterministic shifts]{Least singular values of Gaussian matrices\\with deterministic shifts}

\author{Xuanang Hu}
\address{Shandong University, Jinan, 250100, China}
\email{xuananghu7@gmail.com}

\author{Zeyan Song}
\address{Shandong University, Jinan, 250100, China}
\email{zeyansong8@gmail.com}

\author{Xinglong Wu}
\address{Shandong University, Jinan, 250100, China}
\email{wuxinglong3@gmail.com}

\date{}
\hypersetup{pdftitle={Least singular values of Gaussian matrices with deterministic shifts}}

\begin{document}
\begin{abstract}
Let $A$ be a deterministic real square matrix and $G$ a
standard Gaussian matrix. We determine $\E\smin(A+G)$,
up to universal constants, in terms of the singular
values of $A$. The expectation is comparable to the
square root of the smallest zero of the expected
characteristic polynomial of $(A+G)^{\top}(A+G)$.
This resolves Problem~3 in Tikhomirov's ICM 2022 survey.
We also prove that the zeros of the expected
characteristic polynomial are majorized by the
expected ordered eigenvalues of $(A+G)^{\top}(A+G)$.
\end{abstract}
\maketitle

\section{Introduction}
The least singular value of a randomly perturbed matrix is a central quantity in quantitative invertibility and smoothed analysis \cite{SpielmanTeng,SST,Tikhomirov}. It governs the conditioning of linear systems and hence their numerical stability \cite{Demmel}. Let \(G\) be an \(n\times n\) matrix with independent \(N(0,1)\) entries, and let \(A\in\R^{n\times n}\) be deterministic. Sankar, Spielman, and Teng \cite{SST} established a uniform lower-tail estimate at scale \(n^{-1/2}\): for every \(t>0\),
\[
\Prob \left\{\smin(A+G)\le \frac{t}{\sqrt n}\right\}\le Ct,
\]
where \(C>0\) is an absolute constant. However, this estimate does not capture how the singular values of \(A\) determine the expectation \(\E \smin(A+G)\). Tikhomirov \cite[Section~5, Problem~3]{Tikhomirov} asked for an estimate that reflects this dependence.

Since the squared singular values of \(A+G\) are the eigenvalues of the Gram matrix \(B=(A+G)^{\top}(A+G)\), it is natural to study the expected characteristic polynomial
\[
p_A(x)=\E \det(xI-B).
\]
Its zeros provide deterministic quantities that can be compared with the eigenvalues of \(B\). In the classical Wishart setting, the expected characteristic polynomial is a Laguerre polynomial. Dette \cite{Dette} established strong asymptotic approximations of Wishart eigenvalues by the zeros of appropriately scaled generalized Laguerre polynomials.

Hardy \cite{Hardy} showed that, for certain determinantal point processes, the zeros of the expected characteristic polynomial and the empirical spectral measure have the same limiting distribution. More generally, the theory of finite free convolutions \cite{MSSconv} provides a finite-dimensional framework for expected characteristic polynomials, including bounds on their extreme zeros.

For Gaussian matrices with deterministic shifts, the Gram matrix \(B\) belongs to the class of noncentral Wishart matrices. Classical formulas describe the coefficients of its expected characteristic polynomial \cite{ShahKhatri}; see also \cite{Forrester}. Marcus \cite{Marcus} studied the evolution of the polynomial zeros and showed that their dynamics have the same drift as the Wishart eigenvalue process. His numerical experiments for related
generalized singular-value processes also suggest majorization
phenomena; see \cite[Section~6]{Marcus}.

These connections motivate a quantitative question: how accurately do the zeros of an expected characteristic polynomial describe the spectrum of the underlying random matrix? We address this question in two ways. First, we show that the expected least singular value is comparable, up to universal constants, to the square root of the smallest zero of \(p_A\). Second, we establish a majorization inequality relating all the zeros of \(p_A\) to the expected ordered eigenvalues of the random Gram matrix. In particular, the first result answers Tikhomirov's question, while the second gives a rigorous majorization relation for Gaussian Gram matrices with arbitrary deterministic shifts motivated by Marcus's numerical observations.

Let \(0\le d_1\le\cdots\le d_n\) denote the singular values of \(A\). As shown in Section~2, the zeros of \(p_A\) are positive and simple. We write them as
\[
0<r_1(A)<\cdots<r_n(A)
\]
and set \(\rho(A)=r_1(A)\).

\begin{theorem}\label{thm:main}
Let $A\in\R^{n\times n}$ be deterministic.
Then
\begin{equation}\label{eq:root}
    c\sqrt{\rho(A)}
\le \E\smin(A+G)\le\sqrt{\rho(A)}
\end{equation}
for an absolute constant $c>0$.
Writing $0\le d_1\le\cdots\le d_n$ for the
singular values of $A$, let $\tau=\tau(A)>0$ be the unique solution of
\begin{equation}\label{eq:tau}
\frac1\tau+\sum_{i=1}^n
\frac1{d_i^2+\tau}=1.
\end{equation}
If $A$ is invertible, set
\[
\delta=\left(1-\sum_{i=1}^nd_i^{-2}\right)_+,
\qquad
v=\left(\sum_{i=1}^nd_i^{-4}\right)^{1/2}.
\]
Then
\begin{equation}\label{eq:explicit}
    \E\smin(A+G)\asymp
\tau^{-1/2}
+\delta\min\left\{d_1,\frac{\sqrt\delta}{v}\right\}.
\end{equation}
For singular $A$, the second term is taken to be zero.
\end{theorem}

The upper bound in \eqref{eq:root} is a consequence of
the majorization relation between the zeros of the expected
characteristic polynomial and the expected ordered
eigenvalues of the random Gram matrix. We now state this
relation precisely.

\begin{theorem}\label{thm:majorization}
Let $B=(A+G)^{\top}(A+G)$, and let
$\lambda_1(B)\le\cdots\le\lambda_n(B)$ be its eigenvalues.
For every deterministic $A\in\R^{n\times n}$ and $1\le k\le n$,
\begin{equation}\label{eq:majorization}
 \E\sum_{i=1}^k\lambda_i(B)\le\sum_{i=1}^kr_i(A),
\end{equation}
with equality for $k=n$. Equivalently,
\[
 (r_1(A),\ldots,r_n(A))
 \prec(\E\lambda_1(B),\ldots,\E\lambda_n(B)),
\]
where $\prec$ denotes majorization. Consequently, for every convex
$\varphi:[0,\infty)\to\R$ for which the expectations exist,
\begin{equation}\label{eq:convex-order}
 \sum_{i=1}^n\varphi(r_i(A))\le\E\tr\varphi(B).
\end{equation}
Thus the uniform measure on the zeros of $p_A$ is smaller in convex
order than the expected empirical spectral measure of $B$.
\end{theorem}

The majorization theorem compares partial sums of the two sequences, but does not imply a termwise comparison. This raises the question of whether the relation can be strengthened to individual eigenvalues. 

\begin{conjecture}\label{conj:individual}
There exist absolute constants $c,C>0$ such that,
for every $n\ge1$, every deterministic
$A\in\R^{n\times n}$, and every $1\le k\le n$,
\[
c\,r_k(A)\le
\E \lambda_k(B)
\le C\,r_k(A),
\qquad B=(A+G)^{\top}(A+G).
\]
\end{conjecture}

Motivated by universality phenomena in random matrix theory, we propose a more ambitious conjecture that extends this comparison beyond Gaussian perturbations.

\begin{conjecture}\label{conj:subgaussian}
Let $X$ be an $n\times n$ random matrix with iid
real entries of mean zero and variance one, satisfying
\[
\|X_{11}\|_{\psi_2}\le K.
\]
There exist constants $c_K,C_K>0$, depending only
on $K$, such that, for every deterministic
$A\in\R^{n\times n}$ and every $1\le k\le n$,
\[
c_K r_k(A)
\le \E \lambda_k(B_X)
\le C_K r_k(A),
\qquad B_X=(A+X)^{\top}(A+X).
\]
Here $r_k(A)$ denotes the $k$-th zero of the
Gaussian expected characteristic polynomial $p_A$.
\end{conjecture}

\subsection*{Outline of the proof}
To prove theorem \ref{thm:majorization}, we use
matrix Brownian motion $(W_t)_{t\ge0}$ and exploit the time evolution of the expected characteristic polynomial
\[
p_t(x;a)
=\E\det\left(xI-(A+W_t)^{\top}(A+W_t)\right),
\qquad a=a(A),
\]
where  $a(A)=(d_1^2,\dots,d_n^2)$ denotes the vector of squared singular values
of $A$. By orthogonal invariance, \(p_t(x;a)\) depends on \(A\) only through \(a(A)\). As shown in Lemma~\ref{lem:zeros}, the zeros of
$p_t(x;a)$ are positive and simple for every $t>0$.
Write them as
\[
0<r_1(t;a)<\cdots<r_n(t;a),
\]
and set
\[
F_k(t;a)=\sum_{i=1}^k r_i(t;a).
\]
 In particular,\[
F_k(1;a(A))=\sum_{i=1}^k r_i(A),
\qquad
F_k(0;a(A+W_1))
=\sum_{i=1}^k\lambda_i((A+W_1)^{\top}(A+W_1)).
\]
The key step is to show that
\[
F_k(1-s;a(A+W_s)),\qquad 0\le s<1,
\]
is a supermartingale. This yields the majorization
inequality \eqref{eq:majorization} by taking expectations and passing to
the terminal time.

The key observation is that the zeros of \(p_t(x;a)\) and the squared singular values under matrix Brownian motion have the same drift. Indeed, by \eqref{eq:ode},
\[
\partial_t r_i(t;a)=b_i(r(t;a)),
\qquad
b_i(z)=1+2z_i\sum_{j\ne i}\frac1{z_i-z_j}.
\]
The semigroup property of the zero evolution \eqref{eq:semigroup-property} then gives
\[
\partial_t F_k(t;a)
=\sum_{j=1}^n b_j(a)\partial_{a_j}F_k(t;a).
\]
On the other hand, It\^o's formula for the squared singular values of \(A+W_s\) yields
\[
\begin{aligned}
\frac12\Delta_A F_k(t;a(A))
&=\sum_{j=1}^n b_j(a)\partial_{a_j}F_k(t;a)\\
&\quad+2\sum_{j=1}^n a_j\partial_{a_j}^2F_k(t;a),
\end{aligned}
\]
where \(a=a(A)\). Thus the first-order terms cancel. To control the remaining term, we use the coefficient formula \eqref{eq:coeff}, which shows that \(p_t(x;a)\) is affine in each coordinate of \(a\). A rank-one representation and Ky Fan's variational principle then imply the coordinatewise concavity of \(F_k\) (Lemma~\ref{lem:concavity}). Consequently,
\[
\left(\frac12\Delta_A-\partial_t\right)
F_k(t;a(A))
=2\sum_{j=1}^n a_j\partial_{a_j}^2F_k(t;a)
\le0.
\]
It\^o's formula therefore establishes the supermartingale property.

The lower bound in \eqref{eq:root} requires two different approaches, depending on the relative sizes of \(\delta\) and \(v\). To explain this distinction, assume that \(A\) is invertible and, by orthogonal invariance, write \(A=D=\diag(d_1,\ldots,d_n)\). A natural approach is Gordon's  min--max comparison, which gives
\[
\E\smin(D+G)\ge\E\Phi(g,h),
\qquad
\Phi(g,h)=\min_{\norm{x}=1}
\left(\norm{Dx+g}+\ip{h}{x}\right),
\]
where \(g,h\) are independent standard Gaussian vectors. We use symmetrization to obtain a nonnegative quantity, allowing us to estimate the expectation on a favorable event alone. Indeed, since \((g,h)\stackrel{d}{=}(h,-g)\),
\[
\E\Phi(g,h)
=\frac12\E\bigl[\Phi(g,h)+\Phi(h,-g)\bigr],
\]
where
\(\Phi(g,h)+\Phi(h,-g)\ge0.\)
For \(0<\lambda<a_1\), set
\(R=(D^2-\lambda I)^{-1}.\)
We have
\[
\Phi(g,h)+\Phi(h,-g)
\ge 2\sqrt{\lambda(1-g^{\top}Rg)(1-h^{\top}Rh)}
\]
whenever \(g^{\top}Rg<1\) and \(h^{\top}Rh<1\). Thus the effectiveness of Gordon's comparison depends on whether these Gaussian quadratic forms remain below \(1\). At \(\lambda=0\),
\[
\E g^{\top}D^{-2}g=\sum_{i=1}^n d_i^{-2},
\qquad
\Var(g^{\top}D^{-2}g)=2v^2.
\]
Hence \(\delta\) measures the positive gap between the mean and the threshold \(1\), while \(v\) measures the fluctuations. This naturally leads to two regimes, according to whether \(\delta\) is large or small relative to \(v\).

When \(\delta>32v\), the gap is sufficiently large to control these fluctuations. Gordon's comparison yields
\[
\E\smin(D+G)\gtrsim
\delta\min\left\{d_1,\frac{\sqrt\delta}{v}\right\}.
\]

When \(\delta\le32v\), the gap need not dominate the fluctuations. We instead use a negative-moment argument based on a determinant change of measure. Writing \(X=D+G\), we have
\[
\smin(X)\ge\mathcal T(X)^{-1/2},
\qquad
\mathcal T(X)=\tr((XX^{\top})^{-1}).
\]
Since \(\E\mathcal T(X)\) may be infinite, we seek a bound for the fractional moment \(\E\mathcal T(X)^{1/4}\). The coefficient formula \eqref{eq:coeff} identifies the natural quantity
\[
\kappa=\sum_{i=1}^n\frac1{r_i(A)}
=\frac{\E[\det(XX^{\top})\mathcal T(X)]}
{\E\det(XX^{\top})}.
\]
Thus \(\kappa\) is exactly the expected inverse trace 
\(\E^*\mathcal T(X)\)
under the determinant-weighted Gaussian law
\[
\frac{d\Prob^*}{d\Prob}
=\frac{\det(XX^{\top})}{\E\det(XX^{\top})}.
\]

Our goal is to show that
\[
\E\mathcal T(X)^{1/4}\le 8\kappa^{1/4}.
\]
Indeed, since \(\smin(X)\ge\mathcal T(X)^{-1/2}\), Jensen's inequality then gives
\[
\E\smin(X)\ge\frac1{64\sqrt \kappa}.
\]

To establish this fractional-moment bound, we compare the determinant-weighted square model with the rectangular Gaussian matrix
\[
Y=[D,0_{n\times4}]+G_{n\times(n+4)}.
\]
By Girsanov's theorem, under \(\Prob^*\), the drift of the squared singular values is
\[
b_i(z)+4z_i\partial_i\log H(1-t;z)
\le b_i(z)+4,
\]
where \(H\) is defined in \eqref{eq:H}. The right-hand side is precisely the drift of the squared singular values of the rectangular Gaussian model. The eigenvalue comparison theorem therefore yields
\[
\E\mathcal T(Y)\le \E^*\mathcal T(X)=\kappa.
\]
We then use a random orthogonal projection to transfer this estimate back to the square model. The projection estimate \eqref{eq:projection} gives
\[
\E\mathcal T(X)^{1/4}
\le 8\bigl(\E\mathcal T(Y)\bigr)^{1/4}
\le 8\kappa^{1/4},
\]
as required. Finally, Lemma~\ref{lem:scale} gives \(\rho(A)\kappa\lesssim1\) when \(\delta\le32v\), completing the proof of the lower bound in \eqref{eq:root}.

Section~2 derives the polynomial evolution and proves
Theorem~\ref{thm:majorization}.
Section~3 establishes the negative-moment estimate. Section~4 combines the estimate with Gordon's comparison
and proves Theorem \ref{thm:main}. All vector norms are Euclidean.

\section{The polynomial evolution and majorization}

Let $(W_t)_{t\ge0}$ be an $n\times n$ matrix Brownian motion with
independent standard entries and write $p_t(x;a)$ as in the Introduction.
Let $e_j(a)$ be the $j$th elementary symmetric polynomial, with
$e_0(a)=1$. 
By \cite[Corollary, Eq.~(2.16)]{ShahKhatri},
\begin{equation}\label{eq:coeff}
 p_t(x;a)=\sum_{\ell=0}^n(-1)^\ell x^{n-\ell}
 \sum_{j=0}^\ell
 (\ell-j)!\binom{n-j}{\ell-j}^{\!2}
 t^{\ell-j}e_j(a).
\end{equation}
Differentiating \eqref{eq:coeff} shows that $p_t$
satisfies the evolution equation
\begin{equation}\label{eq:evolution}
\partial_t p_t=-(x\partial_x^2+\partial_x)p_t,
\qquad
p_0(x;a)=\prod_{i=1}^n(x-a_i).
\end{equation}
This equation and its initial condition uniquely
determine $p_t$ among polynomials in $x$ of degree at most $n$.

\begin{lemma}\label{lem:zeros}
For $t>0$, the zeros
$0<r_1(t;a)<\cdots<r_n(t;a)$ of $p_t(x;a)$ are simple.
Each $r_i(t;a)$ is nondecreasing in every coordinate of $a$.
The zeros satisfy
\begin{equation}\label{eq:ode}
 \partial_t r_i=b_i(r),\qquad
 b_i(z)=1+2z_i\sum_{j\ne i}\frac1{z_i-z_j}.
\end{equation}
\end{lemma}

\begin{proof}
For $A=D$, write the rows of $A+W_t$ as $v_i^{\top}$.
They are independent and satisfy
\[
\E v_iv_i^{\top}=tI+a_ie_ie_i^{\top},
\]
where $e_i$ is the $i$th standard basis column vector of
$\R^n$. The mixed characteristic polynomial theorem
\cite[Theorem~4.1 and Corollary~4.4]{MSS}
implies that $p_t(x;a)$ is real-rooted.
Although Theorem~4.1 of \cite{MSS} assumes finite support,
its identity requires only independence and finite second
moments, and hence applies to Gaussian vectors.
Since $A+W_t$ is almost surely invertible for $t>0$,
we have $$(-1)^np_t(x;a)>0$$ for $x\le0$.
Thus all zeros are positive.

Suppose a zero $\alpha>0$ has multiplicity $m\ge2$ at
$t=t_0>0$. Choose a contour $\Gamma$ enclosing $\alpha$
and no other zero of $p_{t_0}(\cdot;a)$. Define
\[
V(t)=\frac1{2\pi i}\int_\Gamma
(x-\alpha)^2\frac{\partial_xp_t(x;a)}{p_t(x;a)}\,dx.
\]
By the residue theorem and real-rootedness,
$V(t)\ge0$ for $t$ near $t_0$, with $V(t_0)=0$.
Writing $p_{t_0}(x)=(x-\alpha)^m q(x)$ with
$q(\alpha)\ne0$, differentiation under the integral
and integration by parts give
\[
\begin{aligned}
V'(t)
&=\frac1{2\pi i}\int_\Gamma
(x-\alpha)^2
\partial_x\left(\frac{\partial_t p_t}{p_t}\right)\,dx\\
&=-\frac1{2\pi i}\int_\Gamma
2(x-\alpha)\frac{\partial_t p_t}{p_t}\,dx\\
&=\frac1{2\pi i}\int_\Gamma
\frac{2(x-\alpha)
(x\partial_x^2p_t+\partial_xp_t)}{p_t}\,dx,
\end{aligned}
\]
where the last equality follows from \eqref{eq:evolution}.
Consequently,
\[
V'(t_0)
=\operatorname*{Res}_{x=\alpha}
\frac{2(x-\alpha)
(x\partial_x^2p_{t_0}+\partial_xp_{t_0})}{p_{t_0}}
=2\alpha m(m-1)>0.
\]
This contradicts the fact that $V$ has a local minimum at $t_0$.

Differentiating $p_t(r_i(t;a);a)=0$ with respect to $t$
and using \eqref{eq:evolution}, we obtain
\[
\begin{aligned}
\partial_t r_i(t;a)
&=-\frac{\partial_t p_t(r_i)}{\partial_xp_t(r_i)}\\
&=1+r_i\frac{\partial_x^2p_t(r_i)}
                  {\partial_xp_t(r_i)}\\
&=1+2r_i\sum_{j\ne i}\frac1{r_i-r_j},
\end{aligned}
\]
where $r_j=r_j(t;a)$. This proves \eqref{eq:ode}.

For $0<z_1<\cdots<z_n$, we have
\[
\partial_{z_j}b_i(z)
=\frac{2z_i}{(z_i-z_j)^2}\ge0
\qquad(i\ne j).
\]
Assume first that $0<a_1<\cdots<a_n$.
For a fixed $\ell$, differentiating \eqref{eq:ode}
with respect to $a_\ell$ gives the variational
equation for $u_i(t)=\partial_{a_\ell}r_i(t;a)$:
\[
u_i'(t)=\sum_{j=1}^n
\partial_{z_j}b_i(r(t;a))u_j(t),
\qquad
u_i(0)=\mathbf1_{\{i=\ell\}},
\]
where the initial condition follows from
$r_i(0;a)=a_i$.
Since the off-diagonal coefficients are nonnegative,
\cite[Theorem~1.1(a)]{Hirsch} gives
\[
\partial_{a_\ell}r_i(t;a)=u_i(t)\ge0
\qquad(1\le i,\ell\le n).
\]
By continuity and symmetry in $a$, the coordinatewise
monotonicity extends to all $a\in[0,\infty)^n$.
\end{proof}

\begin{lemma}\label{lem:concavity}
For $t>0$ and $1\le k\le n$, the function
\[
F_k(t;a)=\sum_{i=1}^k r_i(t;a)
\]
is concave in each coordinate of $a$.
\end{lemma}

\begin{proof}
Fix $t>0$ and vary only the coordinate $a_\ell=u$, writing
$a(u)$ for the resulting vector. By \eqref{eq:coeff},
\[
p_t(x;a(u))=f(x)-ug(x),
\]
where $f$ and $g$ are monic polynomials of degrees
$n$ and $n-1$, independent of $u$.

Let $\alpha_1<\cdots<\alpha_n$ be the zeros of $f$.
Differentiating $p_t(r_i(t;a(u));a(u))=0$
with respect to $u$ at $u=0$ gives
\[
w_i:=\left.\partial_u r_i(t;a(u))\right|_{u=0}
=\frac{g(\alpha_i)}{f'(\alpha_i)}\ge0,
\]
where the inequality follows from Lemma~\ref{lem:zeros}.
Since $f$ and $g$ are monic,
\[
\frac{g(x)}{f(x)}
=\sum_{i=1}^n\frac{w_i}{x-\alpha_i},
\qquad \sum_{i=1}^n w_i=1.
\]

Set $\Lambda=\diag(\alpha_1,\ldots,\alpha_n)$ and
$q=(\sqrt{w_1},\ldots,\sqrt{w_n})^{\top}$.
Then
\[
\begin{aligned}
\det(xI-\Lambda-uqq^{\top})
&=\det(xI-\Lambda)
\left(1-u\,q^{\top}(xI-\Lambda)^{-1}q\right)\\
&=f(x)\left(1-u\sum_{i=1}^n
\frac{w_i}{x-\alpha_i}\right)\\
&=f(x)-ug(x)\\
&=p_t(x;a(u)).
\end{aligned}
\]
Thus the eigenvalues of $\Lambda+uqq^{\top}$, arranged in
increasing order, are precisely
$r_1(t;a(u)),\ldots,r_n(t;a(u))$.
By Ky Fan's variational principle
\cite[Theorem~1]{Fan},
\[
\begin{aligned}
F_k(t;a(u))
&=\sum_{i=1}^k\lambda_i(\Lambda+uqq^{\top})\\
&=\min_{\substack{U\in\R^{n\times k}\\U^{\top}U=I_k}}
\left\{\tr(U^{\top}\Lambda U)+u\|U^{\top}q\|_2^2\right\}.
\end{aligned}
\]
As a minimum of affine functions of $u$,
this is concave.
\end{proof}

\begin{proof}[Proof of Theorem~\ref{thm:majorization}]
Fix $1\le k\le n$. We first show that
\[
Y_s=F_k(1-s;a(A+W_s)),\qquad 0\le s<1,
\]
is a supermartingale.
By \eqref{eq:evolution}, the polynomials
$p_t(x;r(s;a))$ and $p_{t+s}(x;a)$ satisfy
the same evolution equation with the same
initial polynomial $p_s(x;a)$. Hence
\begin{equation}\label{eq:semigroup-property}
    r(t;r(s;a))=r(t+s;a).
\end{equation}
For $0<a_1<\cdots<a_n$, differentiating
$F_k(t;r(s;a))=F_k(t+s;a)$ with respect to $s$
at $s=0$ and using \eqref{eq:ode} gives
\[
\partial_tF_k(t;a)
=\sum_{j=1}^n b_j(a)\partial_{a_j}F_k(t;a).
\]

For distinct positive squared singular values
$a=a(A)$, the generator of the squared singular
values of $A+W_s$ is
\cite[Equation~(1.5)]{KonigOConnell}
\[
\mathcal L
=2\sum_{j=1}^n a_j\partial_{a_j}^2
+\sum_{j=1}^n b_j(a)\partial_{a_j}.
\]
Consequently,
\[
\begin{aligned}
\frac12\Delta_AF_k(t;a(A))
&=\mathcal L F_k(t;a)=\sum_{j=1}^n b_j(a)\partial_{a_j}F_k(t;a)
  +2\sum_{j=1}^n a_j\partial_{a_j}^2F_k(t;a).
\end{aligned}
\]
Combining this with
Lemma~\ref{lem:concavity}, we obtain
\begin{equation}\label{eq:superharmonic}
\left(\frac12\Delta_A-\partial_t\right)
F_k(t;a(A))
=2\sum_{j=1}^n a_j\partial_{a_j}^2F_k(t;a)
\le0.
\end{equation}
For $t>0$, the coefficients of $p_t(x;a(A))$
depend smoothly on $A$, and its zeros are simple.
Thus $F_k(t;a(A))$ is smooth in $(t,A)$, and
the inequality extends to every $A$ by continuity.
By It\^o's formula,
\[
\begin{aligned}
dY_s
={}&\left\langle
\nabla_A F_k(1-s;a(A+W_s)),dW_s
\right\rangle_F\\
&+\left(\frac12\Delta_A-\partial_t\right)
F_k(1-s;a(A+W_s))\,ds.
\end{aligned}
\]
By \eqref{eq:superharmonic}, the drift is
nonpositive. Thus $Y_s$ is a nonnegative local
supermartingale, and hence a supermartingale. Therefore, for
$0<\varepsilon<1$,
\[
\E F_k(\varepsilon;a(A+W_{1-\varepsilon}))
\le F_k(1;a)
=\sum_{i=1}^k r_i(A).
\]
As $\varepsilon\downarrow0$, continuity gives
\[
F_k(\varepsilon;a(A+W_{1-\varepsilon}))
\longrightarrow
\sum_{i=1}^k\lambda_i(B)
\qquad\text{almost surely}.
\]
Fatou's lemma yields
\[
\E\sum_{i=1}^k\lambda_i(B)
\le\sum_{i=1}^k r_i(A).
\]
Moreover, the coefficient of $x^{n-1}$ in
\eqref{eq:coeff} gives
\[
\sum_{i=1}^n r_i(A)
=\sum_{i=1}^n a_i+n^2
=\E\tr B.
\]
Thus equality holds for $k=n$, proving the
majorization relation.

By the standard majorization argument in
\cite[Proposition~3.C.1(a)]{MOA} and Jensen's inequality,
\[
\sum_{i=1}^n\varphi(r_i(A))
\le\sum_{i=1}^n\varphi(\E\lambda_i(B))
\le\E\sum_{i=1}^n\varphi(\lambda_i(B))
\]
for every convex function $\varphi$.
This proves \eqref{eq:convex-order}.
\end{proof}

Taking $k=1$ in Theorem~\ref{thm:majorization} and applying
Cauchy--Schwarz gives the upper bound in Theorem~\ref{thm:main}:
\begin{equation}\label{eq:upper}
 \E\left(\smin(A+G)^2\right)\le\rho(A),
 \qquad \E\smin(A+G)\le\sqrt{\rho(A)}.
\end{equation}

\section{A negative-moment estimate}

For $t\ge0$ and $z\in[0,\infty)^n$, define
\begin{equation}\label{eq:H}
\begin{aligned}
H(t;z)
&=(-1)^n p_t(0;z)\\
&=\sum_{j=0}^n(n-j)!e_j(z)t^{n-j}\\
&=\int_0^\infty e^{-u}
  \prod_{i=1}^n(z_i+tu)\,du.
\end{aligned}
\end{equation}
By \eqref{eq:evolution},
\[
\partial_t H(t;z)
=(-1)^{n-1}\partial_xp_t(0;z).
\]
Write $H(t)=H(t;a)$ and
$\kappa=H'(1)/H(1)$, where the prime
denotes differentiation in $t$.
Since the zeros of $p_1(x;a)$ are
$r_1(A),\ldots,r_n(A)$, we have
\[
\kappa=\sum_{i=1}^n\frac1{r_i(A)}.
\]
For $X=D+G$, we also have
\[
H(1)=\E\det(XX^{\top}),\qquad
H'(1)=\E\tr(\adj(XX^{\top})).
\]
For a matrix $C$ of full row rank, define
$\mathcal T(C)=\tr\left((CC^{\top})^{-1}\right)$.

\begin{lemma}[See Corollary~5 of \cite{GM}]
\label{lem:laguerre}
Let $C$ be a deterministic $n\times(n+k)$ matrix, and let $W_t$ be a standard
matrix Brownian motion. The ordered eigenvalues
$\lambda_1(t)\le\cdots\le\lambda_n(t)$ of
\[
(C+W_t)(C+W_t)^{\top}
\]
satisfy
\[
d\lambda_i
=2\sqrt{\lambda_i}\,d\beta_i
+\left(
1+2\lambda_i\sum_{j\ne i}
\frac1{\lambda_i-\lambda_j}+k
\right)dt,
\]
where $\beta_1,\ldots,\beta_n$ are independent standard
Brownian motions. The corresponding nonnegative strong
solutions are unique and have no explosion or collision
for $t>0$.
\end{lemma}

\begin{proposition}
\label{prop:negative}
There is an absolute constant $c>0$ such that
\[
\E \smin(D+G)\ge\frac{c}{\sqrt \kappa}.
\]
One may take $c=1/64$.
\end{proposition}

\begin{proof}
Write
\[
\mathcal T(X)=\tr\big((XX^{\top})^{-1}\big).
\]
Since
\[
\smin(X)\ge\mathcal T(X)^{-1/2},
\]
Jensen's inequality gives
\[
\E \smin(D+G)
\ge
\left(\E\mathcal T(D+G)^{1/4}\right)^{-2}.
\]
It suffices to prove
\begin{equation}\label{eq:negative-moment}
\E\mathcal T(D+G)^{1/4}\le8\kappa^{1/4}.
\end{equation}

Let $(\widetilde W_t)_{t\ge0}$ be an $n\times(n+4)$ matrix Brownian motion, and set
\[
Y_t=[D,0_{n\times4}]+\widetilde W_t,\qquad Y=Y_1.
\]
We first show that
\begin{equation}\label{eq:rect}
\E\mathcal T(Y)\le \kappa.
\end{equation}

Assume initially that $a_1,\ldots,a_n$ are positive
and distinct. For $X_t=D+W_t$, change measure by
\[
\frac{d\Prob^*}{d\Prob}
=\frac{\det(X_1X_1^{\top})}{H(1)}.
\]
Let $\lambda_1(t)\le\cdots\le\lambda_n(t)$ be the eigenvalues of $X_tX_t^{\top}$.
The conditional density of $\Prob^*$ at time $t<1$ is
\[
M_t=\frac{H(1-t;\lambda(t))}{H(1)}.
\]
Under $\Prob^*$, write $\xi_1(t)\le\cdots\le\xi_n(t)$ for the eigenvalues of $X_tX_t^{\top}$. It\^o's formula and Girsanov's theorem show that the process $\xi(t)$ has drift
\[
b_i(z)+4z_i\partial_i\log H(1-t;z),
\qquad
b_i(z)=1+2z_i\sum_{j\ne i}\frac1{z_i-z_j}.
\]
Since $H(s;z)$ has nonnegative coefficients and is
affine in each $z_i$,
\[
0\le z_i\partial_i\log H(s;z)\le1.
\]
Hence the tilted drift is bounded above by $b_i(z)+4$.

Let $\zeta_1(t)\le\cdots\le\zeta_n(t)$ be the eigenvalues of $Y_tY_t^{\top}$.
By Lemma~\ref{lem:laguerre}, they satisfy the same stochastic
differential equation with drift $b_i+4$.
By the multidimensional comparison theorem for SDEs
with diagonal diffusion and quasi-monotone drift
\cite[Theorem~1.2]{GeissManthey}, the two eigenvalue
processes can be coupled so that
\[
\xi_i(t)\le\zeta_i(t),\qquad i=1,\ldots,n.
\]
Indeed, the comparison condition follows from
\[
0\le 4z_i\partial_i\log H(1-t;z)\le4
\]
and
\[
\partial_{z_j}b_i(z)
=\frac{2z_i}{(z_i-z_j)^2}\ge0,\qquad j\ne i.
\]
We first apply the theorem away from the collision
hyperplanes and then remove the localization.

Consequently, for every $\varepsilon>0$,
\[
\E\sum_{i=1}^n\frac1{\zeta_i(1)+\varepsilon}
\le
\E^*\sum_{i=1}^n\frac1{\xi_i(1)+\varepsilon}.
\]
Letting $\varepsilon\downarrow0$ gives
\[
\begin{aligned}
\E\mathcal T(Y)
&\le \E^*\mathcal T(X_1)=\frac{\E\tr\big(\adj(X_1X_1^{\top})\big)}{H(1)}
=\kappa.
\end{aligned}
\]
Approximation and Fatou's lemma extend
\eqref{eq:rect} to every $D$.

We next establish the projection estimate
\begin{equation}\label{eq:projection}
\E_P\mathcal T(CP)^{1/4}
\le8\mathcal T(C)^{1/4},
\end{equation}
where $C$ is any fixed full-row-rank
$n\times(n+4)$ matrix and $P$ is an independent uniform
$(n+4)\times n$ orthonormal frame.

By orthogonal invariance, take $C=\Delta[I_n,0]$,
with $\Delta$ positive diagonal. Write
\[
Z=\begin{pmatrix}Z_1\\Z_2\end{pmatrix},
\]
where $Z_1$ and $Z_2$ are independent standard Gaussian
matrices of sizes $n\times4$ and $4\times4$. We may
represent
\[
PP^{\top}=I-Z(Z^{\top}Z)^{-1}Z^{\top}.
\]
Since $Z_2$ is invertible almost surely,
\[
\big([I_n,0]PP^{\top}[I_n,0]^{\top}\big)^{-1}
=I_n+Z_1(Z_2^{\top}Z_2)^{-1}Z_1^{\top}.
\]
Set $M=(Z_2^{\top}Z_2)^{-1}$. Conditioning on $Z_2$ gives
\[
\E_{Z_1}\mathcal T(CP)
=\mathcal T(C)(1+\tr M).
\]
By Jensen's inequality,
\[
\E_{Z_1}\mathcal T(CP)^{1/4}
\le\mathcal T(C)^{1/4}(1+\tr M)^{1/4}.
\]

The formula for the diagonal entries of an inverse
Gram matrix gives $M_{jj}\overset{d}=|g|^{-2}$,
where $g\sim N(0,1)$.
Using the subadditivity of $u^{1/4}$,
\[\E(1+\tr M)^{1/4}
\le1+4\E|g|^{-1/2}
<8.\]
This proves \eqref{eq:projection}.

Conditional on $P$, the noise $GP$ is a standard
square Gaussian matrix. Moreover,
\[
[D,0]PP^{\top}[D,0]^{\top}\preceq D^2.
\]
The eigenvalue comparison above, followed by
\eqref{eq:projection}, Jensen's inequality, and
\eqref{eq:rect}, gives
\[
\begin{aligned}
\E\mathcal T(D+G)^{1/4}
\le\E\mathcal T(YP)^{1/4}
\le8\E\mathcal T(Y)^{1/4}
\le8\big(\E\mathcal T(Y)\big)^{1/4}
\le8\kappa^{1/4}.
\end{aligned}
\]
Combining this with \eqref{eq:negative-moment} proves
\[
\E \smin(D+G)\ge\frac1{64\sqrt \kappa}.
\]
\end{proof}

\section{The lower bound}

By orthogonal invariance, we work with $D=\diag(d_1,\ldots,d_n)$.
Write $a_i=d_i^2$, $\rho=\rho(D)$, and $r_i=r_i(1;a)$.
When $D$ is invertible, put $s=\sum_{i=1}^n a_i^{-1}$,
so that $\delta=(1-s)_+$.

\begin{lemma}\label{lem:scale}
We have
\begin{equation}\label{eq:kappa-tau}
 \frac{\tau}{8}\le 1+\kappa\le \tau.
\end{equation}
If $D$ is invertible and $\delta\le32v$, then
\begin{equation}\label{eq:small-gap}
 \kappa\asymp \tau,\qquad \rho \kappa\le69^2.
\end{equation}
If $D$ is invertible and $\delta>32v$, then
\begin{equation}\label{eq:detgap}
 \rho\le3\min\left\{a_1\delta^2,\frac{\delta^3}{v^2}\right\}.
\end{equation}
\end{lemma}

\begin{proof}
Let $U$ have density
\[
 f(u)=\frac{e^{-u}\prod_{i=1}^n(a_i+u)}{H(1)},\qquad u>0,
\]
and write $\mu=\E U$ and $h(u)=\sum_i u/(a_i+u)$.
Differentiating the integral representation of $H$ and integrating
$(uf(u))'=(1-u+h(u))f(u)$ give
$\kappa=\E h(U)$ and $\mu=1+\kappa$.
Since $h$ is concave, Jensen's inequality yields
$\mu\le1+h(\mu)$, or
$\mu^{-1}+\sum_i(a_i+\mu)^{-1}\ge1$.
By \eqref{eq:tau}, this implies $\mu\le \tau$.

For the reverse bound, $f$ is log-concave and
$(\log f)'(\tau)=-1/\tau$. Thus
$f(u+\tau/2)\ge e^{-1/2}f(u)$ for $0<u<\tau/2$, whence
\[
 \Prob(U<\tau/2)\le\sqrt e\,\Prob(U\ge \tau/2),\qquad
 \mu\ge\frac{\tau}{2(1+\sqrt e)}\ge\frac{\tau}{8}.
\]
This proves \eqref{eq:kappa-tau}, without any invertibility assumption.

We now assume that $a_i>0$ for all $i$. The coefficient
formula \eqref{eq:coeff} gives
\[
 \frac{(-1)^np_1(x;a)}{H(1)}
 =\sum_{j=0}^n\frac{(-x)^jH^{(j)}(1)}{(j!)^2H(1)}
 =\prod_{i=1}^n\left(1-\frac{x}{r_i}\right).
\]
Since $\E U^2=2+4\kappa+H''(1)/H(1)$, comparison of the first
two coefficients yields
\begin{equation}\label{eq:moments}
 \sum_{i=1}^n r_i^{-2}
 =\frac{\mu^2-\Var(U)}2\le\frac \kappa\rho.
\end{equation}

We compare this variance with that of
$f_\theta(u)=\frac{\theta^2}{\theta+v}(1+vu)e^{-\theta u}$,
where $\theta>0$ is chosen so that
$\mu=(\theta+2v)/(\theta(\theta+v))$.
As $a_i^{-1}\le v$,
\[
 (\log(f/f_\theta))''(u)
 =-\sum_i(a_i+u)^{-2}+\frac{v^2}{(1+vu)^2}\le0.
\]
The two densities have the same mass and mean. By log-concavity
of their ratio, unless they coincide their difference has two
crossings $\alpha<\beta$, with signs $-,+,-$. Integrating
$(u-\alpha)(u-\beta)(f-f_\theta)\le0$ therefore shows that
$\Var(U)\le\Var_{f_\theta}(u)$.
A direct computation gives
$(\mu^2-\Var_{f_\theta}(u))/2=\mu^2/(2+\theta/v)^2$;
moreover, the mean equation implies $\theta\le v+\mu^{-1}$.
Together with \eqref{eq:moments}, these facts yield
\begin{equation}\label{eq:rho-mu}
 \rho\le\frac \kappa{\mu^2}\left(3+\frac1{v\mu}\right)^2.
\end{equation}

Set
$w(u)=e^{-\delta u}\prod_i e^{-u/a_i}(1+u/a_i)$ and
$Z=\int_0^\infty w(u)\,du$.
The elementary bounds
\begin{equation}\label{eq:log-bounds}
 z-\frac{z^2}{2}\le\log(1+z)\le z
\end{equation}
imply
$e^{-\delta u-v^2u^2/2}\le w(u)\le e^{-\delta u}\le1$.
Jensen's inequality under an exponential law of rate $\delta+v$
gives $Z\ge(\delta+v)^{-1}$.
Since $f/(w/Z)$ is nondecreasing and $w/Z\le1/Z$,
the mean of $f$ is at least that of $w/Z$, which is at least
$Z/2$. Consequently,
\begin{equation}\label{eq:mu-lower}
 \mu\ge\frac1{2(\delta+v)}.
\end{equation}
If $\delta\le32v$, \eqref{eq:rho-mu}, \eqref{eq:mu-lower}
and $\kappa<\mu$ give $\rho \kappa\le(5+2\delta/v)^2\le69^2$.
Also $v\le s$ implies $s\ge1/33$. Since $f(u)/e^{-u}$ is increasing,
$\mathbb P(U\ge1)\ge e^{-1}$;
therefore
\[
 \kappa=\E h(U)\ge e^{-1}h(1)
 \ge\frac{s}{e(1+s)}\ge\frac1{34e}.
\]
The comparison $\kappa\asymp \tau$ follows from \eqref{eq:kappa-tau},
proving \eqref{eq:small-gap}.

Now suppose $\delta>32v$, so $s<1$ and $\delta=1-s$.
For $t>0$, the identity
$\sum_j r_j(t;a)^{-1}=H'(t)/H(t)$ and \eqref{eq:ode}
give $r_1'(t)+2(H'(t)/H(t))r_1(t)\le3$.
Integrating this inequality yields
\begin{equation}\label{eq:integrate}
 \rho\le\frac{a_1H(0)^2+3\int_0^1H(t)^2\,dt}{H(1)^2},
\end{equation}
first for distinct $a_i$ and then by continuity.
The bounds in \eqref{eq:log-bounds} also imply
\[
 \frac{H(t)}{H(0)}\le\frac1{1-ts},\qquad
 \frac{H(1)}{H(0)}\ge
 \frac1\delta\left(1-\frac{v^2}{\delta^2}\right)
 \ge\frac3{4\delta}.
\]
Using $\int_0^1(1-ts)^{-2}\,dt=1/\delta$ and
$a_1\delta>32$ in \eqref{eq:integrate}, we obtain
$\rho\le2a_1\delta^2$.
Finally, $f=w/Z$ and $Z\le1/\delta$ in this regime, so
\[
 \mu\ge\frac1\delta\left(1-\frac{3v^2}{\delta^2}\right)
 \ge\frac3{4\delta}.
\]
Substituting this and $\kappa<\mu$ into \eqref{eq:rho-mu}
gives $\rho\le3\delta^3/v^2$, proving \eqref{eq:detgap}.
\end{proof}

\begin{proposition}
\label{prop:gordon}
Let $0<\lambda<a_1$, define
$R=(D^2-\lambda I)^{-1}$ and $\eta=1-\tr R$.
If $\eta\ge8(\tr R^2)^{1/2}$, then
\begin{equation}\label{eq:gordon}
 \E\smin(D+G)\ge\frac38\sqrt\lambda\,\eta.
\end{equation}
\end{proposition}
\begin{proof}
Let $g,h\in\R^n$ be independent standard Gaussian
vectors, and define
\[
\Phi(g,h)=\min_{\norm{x}=1}
\left(\norm{Dx+g}+\ip{h}{x}\right).
\]
We first show that
\begin{equation}\label{eq:comparison}
\E\smin(D+G)\ge\E\Phi(g,h).
\end{equation}
Let $\gamma$ be an independent standard Gaussian
scalar. For unit vectors $x,y$, consider
\[
Y_{x,y}=\ip{y}{Gx}+\gamma,
\qquad
X_{x,y}=\ip{g}{y}+\ip{h}{x}.
\]
Both processes are centered with variance $2$, and
\[
\begin{aligned}
&\E Y_{x,y}Y_{x',y'}
 -\E X_{x,y}X_{x',y'}\\
&\qquad=
(1-\ip{x}{x'})(1-\ip{y}{y'})\ge0,
\end{aligned}
\]
with equality whenever $x=x'$.
Gordon's Gaussian min--max comparison
\cite[Theorem~A.1]{TOH}, with the common
deterministic term $\ip{y}{Dx}$, therefore gives
\[
\begin{aligned}
&\E\min_{\norm{x}=1}\max_{\norm{y}=1}
\left(\ip{y}{Dx}+Y_{x,y}\right)\\
&\qquad\ge
\E\min_{\norm{x}=1}\max_{\norm{y}=1}
\left(\ip{y}{Dx}+X_{x,y}\right).
\end{aligned}
\]
Since $\E\gamma=0$, this proves
\eqref{eq:comparison}.

Next, since $(g,h)\stackrel{d}{=}(h,-g)$,
\[
\E\Phi(g,h)
=\frac12\E\bigl[\Phi(g,h)+\Phi(h,-g)\bigr].
\]
For any unit vectors $x,y$,
\[
\begin{aligned}
&\norm{Dx+g}+\ip{h}{x}
 +\norm{Dy+h}-\ip{g}{y}\\
&\quad\ge
\ip{y}{Dx+g}+\ip{h}{x}
-\ip{x}{Dy+h}-\ip{g}{y}=0,
\end{aligned}
\]
where we used $D=D^{\top}$. Taking the minimum
over $x$ and $y$ gives
\begin{equation}\label{eq:sym-nonnegative}
\Phi(g,h)+\Phi(h,-g)\ge0.
\end{equation}

Set
\[
\alpha=1-g^{\top}Rg,\qquad
\beta=1-h^{\top}Rh,
\qquad
\mathcal E=\{\alpha,\beta\ge\eta/2\}.
\]
For $\norm{x}=1$, the identity
$R^{-1}=D^2-\lambda I$ gives
\[
\begin{aligned}
\norm{Dx+g}^2
&=\norm{R^{-1/2}(x+RDg)}^2
  +\lambda\alpha,\\
\ip{h}{x}
&=\ip{R^{1/2}h}{R^{-1/2}(x+RDg)}
  -h^{\top}DRg.
\end{aligned}
\]
On $\mathcal E$, we have $\alpha>0$ and
$\norm{R^{1/2}h}^2=1-\beta<1$.
By Cauchy--Schwarz in $\R^{n+1}$,
\[
\norm{Dx+g}+\ip{h}{x}
\ge\sqrt{\lambda\alpha\beta}-h^{\top}DRg.
\]
Taking the minimum over $\norm{x}=1$ yields
\[
\Phi(g,h)
\ge\sqrt{\lambda\alpha\beta}-h^{\top}DRg.
\]
Applying the same inequality to $(h,-g)$ gives
\[
\Phi(h,-g)
\ge\sqrt{\lambda\alpha\beta}+h^{\top}DRg.
\]
Consequently, on $\mathcal E$,
\[
\Phi(g,h)+\Phi(h,-g)
\ge2\sqrt{\lambda\alpha\beta}
\ge\sqrt\lambda\,\eta.
\]

For a standard Gaussian vector $z$,
\[
\E z^{\top}Rz=\tr R,\qquad
\Var(z^{\top}Rz)=2\tr R^2.
\]
Chebyshev's inequality and a union bound give
\[
\Prob(\mathcal E^c)
\le\frac{16\tr R^2}{\eta^2}
\le\frac14.
\]
Together with \eqref{eq:comparison} and
\eqref{eq:sym-nonnegative}, this yields
\[
\begin{aligned}
\E\smin(D+G)
&\ge\frac12\E\bigl[\Phi(g,h)+\Phi(h,-g)\bigr]\\
&\ge\frac12\sqrt\lambda\,\eta\,
     \Prob(\mathcal E)\\
&\ge\frac38\sqrt\lambda\,\eta.
\end{aligned}
\]
\end{proof}

\begin{proof}[Proof of Theorem~\ref{thm:main}]
By orthogonal invariance, it suffices to consider $D$.
The upper bound in \eqref{eq:root} follows from \eqref{eq:upper}.
Assume first that $D$ is invertible.

If $\delta\le32v$, Proposition~\ref{prop:negative} and
Lemma~\ref{lem:scale} give
\[
 \E\smin(D+G)\ge\frac1{64\sqrt \kappa}
 \ge\frac{\sqrt\rho}{64\cdot69}.
\]
Moreover, $\kappa\asymp \tau$ by \eqref{eq:small-gap}; hence
\eqref{eq:upper}, \eqref{eq:small-gap}, and the negative-moment
bound imply $\E\smin(D+G)\asymp \tau^{-1/2}$.
By \eqref{eq:tau}, $\delta\le \tau^{-1}$, so
$\delta\min\{d_1,\sqrt\delta/v\}
 \le\delta^{3/2}/v\le32\tau^{-1/2}$.
This proves \eqref{eq:explicit} in this case.

Suppose now that $\delta>32v$, and choose
\[
 \lambda=\min\left\{\frac{a_1}2,\frac{\delta}{4v^2}\right\},
 \qquad R=(D^2-\lambda I)^{-1},\qquad \eta=1-\tr R.
\]
As $a_i-\lambda\ge a_i/2$, we have
$(\tr R^2)^{1/2}\le2v$ and
$\eta\ge\delta-2\lambda v^2\ge\delta/2>16v$.
Proposition~\ref{prop:gordon} and \eqref{eq:detgap} therefore give
\[
 \E\smin(D+G)
 \ge\frac3{32}\delta\min\left\{d_1,\frac{\sqrt\delta}{v}\right\}
 \ge\frac3{32\sqrt3}\sqrt\rho.
\]
In addition, \eqref{eq:tau} yields
$\delta\le \tau^{-1}\le\delta+v^2/\delta\le2\delta$.
Since $a_1\ge1/v$ and $\delta>32v$,
\[
 \delta^2\min\left\{a_1,\frac\delta{v^2}\right\}
 \ge32\delta\ge16\tau^{-1}.
\]
Thus the second term in \eqref{eq:explicit} dominates
$\tau^{-1/2}$. The matching upper bound follows from
\eqref{eq:upper} and \eqref{eq:detgap}, completing the
proof for invertible $D$.

If $D$ is singular, apply the estimates to
$D+\varepsilon I$ and let $\varepsilon\downarrow0$.
The expectation changes by at most $\varepsilon$,
while $\rho$ and $\tau$ depend continuously on $a$.
Since the smallest diagonal entry tends to zero,
$\delta=0$ for all sufficiently small $\varepsilon$;
thus the second term in \eqref{eq:explicit} vanishes.
\end{proof}

\section*{Acknowledgments}
The authors thank Professor Hanchao Wang for his guidance and
helpful suggestions.

\par\medskip
\noindent\textbf{Statement on A.I. use.}
The authors used OpenAI's GPT-5.6 to improve
the language and readability of the manuscript.
The authors take full responsibility for its content.


\begin{thebibliography}{9}

\bibitem{GM}
P.~Graczyk and J.~Ma{\l}ecki, Strong solutions of non-colliding
particle systems, \emph{Electron. J. Probab.} \textbf{19} (2014),
paper no.~119, 21~pp.
\href{https://doi.org/10.1214/EJP.v19-3842}{doi:10.1214/EJP.v19-3842}.
\MR{3296535}

\bibitem{Marcus}
A.~W.~Marcus, Finite free point processes, preprint, 2022.
\href{https://arxiv.org/abs/2205.00495}{arXiv:2205.00495}.

\bibitem{MSS}
A.~W.~Marcus, D.~A.~Spielman and N.~Srivastava, Interlacing
families II: Mixed characteristic polynomials and the Kadison--Singer
problem, \emph{Ann. of Math.} (2) \textbf{182} (2015), no.~1,
327--350.
\href{https://doi.org/10.4007/annals.2015.182.1.8}
{doi:10.4007/annals.2015.182.1.8}.
\MR{3374963}

\bibitem{MSSconv}
A.~W.~Marcus, D.~A.~Spielman and N.~Srivastava, Finite free
convolutions of polynomials, \emph{Probab. Theory Related Fields}
\textbf{182} (2022), nos.~3--4, 807--848.
\href{https://doi.org/10.1007/s00440-021-01105-w}
{doi:10.1007/s00440-021-01105-w}.
\MR{4408504}

\bibitem{SST}
A.~Sankar, D.~A.~Spielman and S.-H.~Teng, Smoothed analysis of
the condition numbers and growth factors of matrices,
\emph{SIAM J. Matrix Anal. Appl.} \textbf{28} (2006), no.~2,
446--476.
\href{https://doi.org/10.1137/S0895479803436202}
{doi:10.1137/S0895479803436202}.

\bibitem{Tikhomirov}
K.~Tikhomirov, Quantitative invertibility of non-Hermitian random
matrices, in \emph{ICM---International Congress of Mathematicians.
Vol.~4. Sections 5--8}, EMS Press, Berlin, 2023, 3292--3313.
\href{https://doi.org/10.4171/ICM2022/10}{doi:10.4171/ICM2022/10}.

\bibitem{TOH}
C.~Thrampoulidis, S.~Oymak and B.~Hassibi, The Gaussian min-max
theorem in the presence of convexity, preprint, 2015.
\href{https://arxiv.org/abs/1408.4837v2}{arXiv:1408.4837v2}.
\bibitem{Hirsch}
M.~W.~Hirsch, Systems of differential equations
that are competitive or cooperative. II.
Convergence almost everywhere,
\emph{SIAM J. Math. Anal.} \textbf{16}
(1985), no.~3, 423--439.
\bibitem{Fan}
K.~Fan, On a theorem of Weyl concerning eigenvalues
of linear transformations. I,
\emph{Proc. Natl. Acad. Sci. USA}
\textbf{35} (1949), no.~11, 652--655.
\href{https://doi.org/10.1073/pnas.35.11.652}
{doi:10.1073/pnas.35.11.652}.
\bibitem{MOA}
A.~W. Marshall, I.~Olkin, and B.~C. Arnold,
\emph{Inequalities: Theory of Majorization and Its Applications},
2nd ed., Springer, New York, 2011.
\bibitem{KonigOConnell}
W.~K\"onig and N.~O'Connell,
Eigenvalues of the Laguerre process as
non-colliding squared Bessel processes,
\emph{Electron. Commun. Probab.}
\textbf{6} (2001), 107--114.
\bibitem{ShahKhatri}
B.~K.~Shah and C.~G.~Khatri,
Proof of conjectures about the expected values
of the elementary symmetric functions of a noncentral Wishart matrix,
\emph{Ann. Statist.} \textbf{2} (1974), no.~4, 833--836.
\href{https://doi.org/10.1214/aos/1176342773}
{doi:10.1214/aos/1176342773}.
\bibitem{Dette}
H.~Dette,
Strong approximation of eigenvalues of large
dimensional Wishart matrices by roots of
generalized Laguerre polynomials,
\emph{J. Approx. Theory}
\textbf{118} (2002), no.~2, 290--304.
\href{https://doi.org/10.1006/jath.2002.3725}
{doi:10.1006/jath.2002.3725}.

\bibitem{Hardy}
A.~Hardy,
Average characteristic polynomials of
determinantal point processes,
\emph{Ann. Inst. Henri Poincar\'e Probab. Stat.}
\textbf{51} (2015), no.~1, 283--303.
\href{https://doi.org/10.1214/13-AIHP572}
{doi:10.1214/13-AIHP572}.

\bibitem{Forrester}
P.~J.~Forrester,
The averaged characteristic polynomial for
the Gaussian and chiral Gaussian ensembles
with a source,
\emph{J. Phys. A: Math. Theor.}
\textbf{46} (2013), no.~34, 345204.
\href{https://doi.org/10.1088/1751-8113/46/34/345204}
{doi:10.1088/1751-8113/46/34/345204}.
\bibitem{GeissManthey}
C.~Gei{\ss} and R.~Manthey,
Comparison theorems for stochastic differential equations
in finite and infinite dimensions,
\emph{Stochastic Process. Appl.}
\textbf{53} (1994), no.~1, 23--35.
\href{https://doi.org/10.1016/0304-4149(94)90055-8}
{doi:10.1016/0304-4149(94)90055-8}.
\bibitem{Demmel}
J.~W. Demmel, The probability that a numerical analysis problem
is difficult, \emph{Math. Comp.} \textbf{50} (1988), no.~182,
449--480.
\href{https://doi.org/10.1090/S0025-5718-1988-0929546-7}
{doi:10.1090/S0025-5718-1988-0929546-7}.
\bibitem{SpielmanTeng}
D.~A. Spielman and S.-H. Teng, Smoothed analysis of algorithms:
Why the simplex algorithm usually takes polynomial time,
\emph{J. ACM} \textbf{51} (2004), no.~3, 385--463.
\href{https://doi.org/10.1145/990308.990310}
{doi:10.1145/990308.990310}.
\end{thebibliography}
\end{document}